\documentclass[11pt,oneside]{amsart}
\usepackage[top=25mm, bottom=25mm, left=20mm, right=20mm]{geometry}
\usepackage{amsfonts,amssymb,amsmath,amsthm,amsxtra}
\usepackage[T1]{fontenc}
\usepackage[utf8]{inputenc}
\usepackage{bm}
\usepackage{mathtools}
\usepackage{dsfont}
\usepackage{mathrsfs}
\usepackage{enumitem}
\allowdisplaybreaks

\usepackage{graphics}
\usepackage{hyperref}

\numberwithin{equation}{section}
\newtheorem{theorem}[equation]{Theorem}
\newtheorem{corollary}[equation]{Corollary}
\newtheorem{lemma}[equation]{Lemma}
\newtheorem{proposition}[equation]{Proposition}
\theoremstyle{definition}
\newtheorem{remark}[equation]{Remark}
\newtheorem{definition}[equation]{Definition}

\DeclareMathOperator\supp{supp}
\author{Leonidas Daskalakis} 
\address[Leonidas Daskalakis]{Institute of Mathematics,
Polish Academy of Sciences,
\'Sniadeckich 8,
00-656 Warszawa, Poland}
\email{ldaskalakis@impan.pl}
\begin{document}
\title[Bilinear ergodic theorems, Roth's Theorem and Corners along fractional powers]{Ergodic theorems for bilinear averages, Roth's Theorem and Corners along fractional powers}
\begin{abstract}We prove that for every $c\in(1,23/22)$, every probability space $(X,\mathcal{B},\mu)$ equipped with two commuting measure-preserving transformations $T,S\colon X\to X$ and every $f,g\in L^{\infty}_{\mu}(X)$ we have that the $L^2_{\mu}(X)$-limit
\[
\lim_{N\to\infty}\frac{1}{N}\sum_{n=1}^Nf(T^{\lfloor n^c\rfloor}x)g(S^{\lfloor n^c\rfloor}x)
\]
equals the $L^2_{\mu}(X)$-limit $\lim_{N\to\infty}\frac{1}{N}\sum_{n=1}^Nf(T^{n}x)g(S^{n}x)$. The approach is based on the author's recently developed technique which may be thought of as a change of variables. We employ it to establish several new results along fractional powers including a Roth-type result for patterns of the form $x,x+\lfloor y^c \rfloor,x+2\lfloor y^c \rfloor$ as well as its ``corner'' counterpart. The quantitative nature of the former result allows us to recover the analogous one in the primes. Our considerations give partial answers to Problem~29 and Problem~30 from Frantzikinakis' open problems survey on multiple ergodic averages \cite{OpProb}. Notably, we cover more general sparse orbits $(\lfloor h(n)\rfloor)_{n\in\mathbb{N}}$, where $h$ belongs to the class of the so-called $c$-regularly varying functions, addressing for example even the orbit $(\lfloor n\log n\rfloor)_{n\in\mathbb{N}}$.
\end{abstract}
\maketitle

\section{Introduction}
The first main result of the present work is the following bilinear ergodic theorem with two commuting transformations along fractional powers, providing a partial answer to Problem~29 from Frantzikinakis' open problems survey on multiple ergodic averages \cite{OpProb}.
\begin{theorem}\label{Ncsimple}Assume $c\in(1,23/22)$, and let $(X,\mathcal{B},\mu)$ be a probability space and $T,S$ measure-preserving transformations on $X$ which commute. Then for every $f,g\in L^{\infty}_{\mu}(X)$ we have that the $L^2_{\mu}(X)$-limit
\begin{equation}\label{nonconvaveasy}
\lim_{N\to\infty}\frac{1}{N}\sum_{n=1}^Nf(T^{\lfloor n^c\rfloor}x)g(S^{\lfloor n^c\rfloor}x)
\end{equation}
equals the $L^2_{\mu}(X)$-limit $\lim_{N\to\infty}\frac{1}{N}\sum_{n=1}^Nf(T^{n}x)g(S^{n}x)$.
\end{theorem} 
In the sequel, we work with more general sparse orbits, derived from the so-called $c$-regularly varying functions. Before stating the rest of our main theorems, we introduce this wide class of functions below.
\begin{definition}[$c$-regularly varying functions] Fix $c\in[1,2)$ and $x_0\ge 1$. We define the class of $c$-regularly varying functions $\mathcal{R}_c$ as the set of all functions $h\in\mathcal{C}^3\big([x_0,\infty)\to[1,\infty)\big)$ such that the following conditions hold:
\begin{itemize}
\item[i)]$h'>0$, $h''> 0$\text{ and }
\item[ii)] the function $h$ is of the form 
\[
h(x)=Cx^c\exp\bigg(\int_{x_0}^x\frac{\vartheta(t)}{t}dt\bigg)\text{,}
\]
where $C$ is a positive constant and $\vartheta\in\mathcal{C}^3\big([x_0,\infty)\to\mathbb{R}\big)$ satisfies
\[
\vartheta(x)\to 0\text{,}\quad x\vartheta'(x)\to 0\text{,}\quad x^2\vartheta''(x)\to 0\text{,}\quad x^3\vartheta'''(x)\to 0\quad\text{as}\quad x\to \infty\text{.}
\]
\item[iii)] If $c=1$, then $\vartheta$ will additionally be assumed to be positive, decreasing, and for every $\varepsilon>0$
\[
\vartheta(x)^{-1}\lesssim_{\varepsilon}x^{\varepsilon}\quad\text{and}\quad \lim_{x\to\infty}xh(x)^{-1}=0\text{.}
\]
Moreover, for $c=1$ we assume
\[
\frac{x\vartheta'(x)}{\vartheta(x)}\to 0\text{,}\quad\frac{x^2\vartheta''(x)}{\vartheta(x)}\to 0\text{,}\quad\frac{x^3\vartheta'''(x)}{\vartheta(x)}\to 0\quad \text{as}\quad x\to\infty\text{.}
\]
\end{itemize}
\end{definition}
The family of $c$-regularly varying functions has been introduced in \cite{MMR} and \cite{WT11}, and one may think of functions in $\mathcal{R}_c$ as appropriate perturbations of the fractional monomial $x^c$, notably, including the following functions for $c>1$
\begin{equation}\label{examples}
x^c\log^{a}x\text{,}\quad x^ce^{a\log^b x}\text{,}\quad x^c\underbrace{\log\circ \cdots \circ \log x}_{k \text{ times}}\text{,}\quad\text{for any $a\in\mathbb{R}$, $b\in(0,1)$, $k\in\mathbb{N}$ fixed.}
\end{equation}
For $c=1$ the examples above do belong in $\mathcal{R}_1$ provided that $a>0$. Let us note that there exist $c$-regularly varying functions which do not belong to any Hardy field since there exist $h\in\mathcal{R}_c$ which are not in $\mathcal{C}^{\infty}$.

We immediately derive Theorem~$\ref{Ncsimple}$ by establishing the following theorem, covering simultaneously all the orbits induced by any of the $c$-regularly varying functions from $\eqref{examples}$.
\begin{theorem}\label{FRA}Assume $c\in[1,23/22)$ and $h\in\mathcal{R}_c$. Let $(X,\mathcal{B},\mu)$ be a probability space and $T,S$ measure-preserving transformations on $X$ which commute. Then for every $f,g\in L^{\infty}_{\mu}(X)$ we have that the $L^2_{\mu}(X)$-limit
\begin{equation}\label{nonconvav}
\lim_{N\to\infty}\frac{1}{N}\sum_{n=1}^Nf(T^{\lfloor h(n)\rfloor}x)g(S^{\lfloor h(n)\rfloor}x)
\end{equation}
equals the $L^2_{\mu}(X)$-limit $\lim_{N\to\infty}\frac{1}{N}\sum_{n=1}^Nf(T^{n}x)g(S^{n}x)$.
\end{theorem}
Our approach for proving Theorem~$\ref{FRA}$ involves certain quantitative estimates for operators on functions on the integers, which, readily give rise to certain results in additive combinatorics. In the sequel, we establish the following Roth-type result for subsets of the natural numbers lacking patterns of the form 
\[
n,n+\lfloor  h(m)\rfloor,n+2\lfloor h(m)\rfloor\text{,}\quad n,m\in\mathbb{N}\text{.}
\]
Notably, our proof is not only quantitative, but also provides an estimate of the same strength as the state of the art bound for the conventional Roth's theorem \cite{BloomSharp}, see Theorem~$\ref{SOTAest}$ and Remark~$\ref{RemarkRoth}$.
\begin{theorem}\label{RothFra}
Assume $c\in[1,23/22)$ and $h\in\mathcal{R}_c$. Then there exist positive constants $\chi=\chi_h$ and $C=C_h$ such that for all $N\in\mathbb{N}$ and $A\subseteq \{1\dotsc,N\}$ lacking patterns of the form $n,n+\lfloor h(m)\rfloor,n+2\lfloor h(m)\rfloor$, $n,m\in\mathbb{N}$ we have
\begin{equation}\label{RothBound}
\frac{|A|}{N}\le C_h\exp\big(-\chi_h(\log N)^{1/9}\big)\text{.}
\end{equation}
\end{theorem}
This immediately implies that if $A\subseteq \mathbb{N}$ has positive upper Banach density, then it contains infinitely many patterns of the form $n,n+\lfloor h(m)\rfloor,n+2\lfloor h(m)\rfloor$. The bounds above are strong enough to yield the corresponding Roth-type theorem for fractional powers in the primes, providing a partial answer to Problem~30 in \cite{OpProb}.
\begin{corollary}
\label{RothFraPrimes}
Fix $c\in[1,23/22)$, $h\in\mathcal{R}_c$ and assume $A\subseteq \mathbb{P}$ has positive upper relative density,
\[
\text{i.e. }\limsup_{N\to\infty}\frac{|A\cap[1,N]|}{|\mathbb{P}\cap[1,N]|}>0\text{.}
\]
Then $A$ contains infinitely many patterns of the form $n,n+\lfloor h(m)\rfloor,n+2\lfloor h(m)\rfloor$, $n,m\in\mathbb{N}$.
\end{corollary}
In a very recent breakthrough \cite{cornerpaper} quasipolynomial bounds are established for corner-free subsets of $\{1,\dotsc,N\}^2$. In the same  spirit as Theorem~$\ref{RothFra}$, we establish the fractional-power variant of the result for corners and give bounds of the same ``shape''. More precisely we obtain the following.
\begin{theorem}\label{corners}
Assume $c\in[1,23,22)$ and $h\in\mathcal{R}_c$. Then there exist positive constants $\chi=\chi_h$ and $C=C_h$ such that for all $N\in\mathbb{N}$ and $A\subseteq \{1\dotsc,N\}^2$ lacking patterns of the form 
\begin{equation}\label{cornerconfig}
(n_1,n_2),(n_1+\lfloor h(m)\rfloor,n_2),(n_1,n_2+\lfloor h(m)\rfloor)\text{,}\quad\text{ $n_1,n_2,m\in\mathbb{N}$,}
\end{equation}
we have
\begin{equation}\label{CornerBound}
\frac{|A|}{N^2}\le C_h\exp\big(-\chi_h(\log N)^{1/600}\big)\text{.}
\end{equation}
\end{theorem}

The proofs of Theorem~$\ref{FRA}$, Theorem~$\ref{RothFra}$ and Theorem~$\ref{corners}$ are very similar in spirit, following the approach from \cite{NcminusNc}. Loosely speaking, a sophisticated ``change of variables'' will be performed, allowing one to pass from summation over the thin set $\mathbb{N}_h\cap[1,N]\coloneqq \{\lfloor h(k)\rfloor:\,k\in\mathbb{N}\}\cap[1,N]$ to summation over $\mathbb{N}\cap[1,N]$. This change of variables can be carried out whenever the averaging operators involved can be appropriately bounded by $U^3$-norms. More precisely, one may employ the strategy we lay out in the present work whenever the corresponding averaging operators, when weighted, can be suitably bounded by the $U^3$-norm of the weight, see for example Lemma~$\ref{U3forFRA}$ in our section~3, or Lemma~3.3 in \cite{NcminusNc}. For more comments on the method we refer the reader to subsection~$\ref{strategy}$. This approach can for example yield conditional results on the $\mu$-a.e. convergence of the ergodic averages $\eqref{nonconvaveasy}$ provided that $\lim_{N\to\infty}\frac{1}{N}\sum_{n=1}^Nf(T^{n}x)g(S^{n}x)$ converges for $\mu$-a.e. $x\in X$. Although we do not give the most general results derivable by our methods, we note that novel results may be established even for single averages. Specifically, we prove the following, which to the best of the author's knowledge, is also new.   
\begin{theorem}\label{n2pt} Assume $c\in[1,23/22)$ and $h\in\mathcal{R}_c$. Let $(X,\mathcal{B},\mu)$ be a probability space and $T$ a measure-preserving transformation on $X$. Then for every $f\in L_{\mu}^{\infty}(X)$ we have
\begin{equation}\label{ptconv}
\lim_{N\to\infty}\frac{1}{N}\sum_{n=1}^Nf(T^{\lfloor h(n)\rfloor^2} x)=\lim_{N\to\infty}\frac{1}{N}\sum_{n=1}^Nf(T^{n^2} x)\quad\text{for $\mu$-a.e. $x\in X$.}
\end{equation}
\end{theorem}
\subsection{Strategy}\label{strategy}Our treatment is similar to the one from \cite{NcminusNc}, see subsection~1.1. Here we wish to present a slightly more general framework allowing us to unify our approach to all of our main theorems. Let $(X,\mathcal{B},\mu)$ be a probability space and let $\mathfrak{f}=(f_n)_{n\in\mathbb{N}}$ be a sequence of $1$-bounded functions on $X$. We consider
\[
B_N\mathfrak{f}(x)\coloneqq\frac{1}{|\mathbb{N}_h\cap [N]|}\sum_{n\in[N]}1_{\mathbb{N}_h}(n)f_n(x)\text{,}
\] 
where $[N]\coloneqq \{1,\dotsc,N\}$ and $\mathbb{N}_h\coloneqq\{\lfloor h(n)\rfloor:\,n\in\mathbb{N}\}$. After exploiting the following formula $1_{\mathbb{N}_h}(n)=\lfloor -\varphi(n)\rfloor-\lfloor -\varphi(n+1)\rfloor$, where $\varphi$ is the compositional inverse of $h$, we may write
\[
1_{\mathbb{N}_h}(n)=\big(\varphi(n+1)-\varphi(n)\big)+\big(\Phi(-\varphi(n+1))-\Phi(-\varphi(n))\big)\text{,}
\]
where $\Phi(x)=\{x\}-1/2$. This decomposes our averaging operators to $B_N\mathfrak{f}=M_N \mathfrak{f}+E_N\mathfrak{f}$, where 
\begin{equation}\label{MN}
M_N\mathfrak{f}(x)\coloneqq \frac{1}{|\mathbb{N}_h\cap [N]|}\sum_{n\in[N]}\big(\varphi(n+1)-\varphi(n)\big)f_n(x)\text{,}
\end{equation}
and
\begin{equation}\label{EN}
E_N\mathfrak{f}(x)\coloneqq\frac{1}{|\mathbb{N}_h\cap [N]|}\sum_{n\in[N]}\big(\Phi(-\varphi(n+1))-\Phi(-\varphi(n))\big)f_n(x)=\mathbb{E}_{n\in[N]}f_n(x)w_N(n)\text{,}
\end{equation}
with $w_{N}(n)\coloneqq\frac{N(\Phi(-\varphi(n+1))-\Phi(-\varphi(n)))}{|\mathbb{N}_h\cap [N]|}$. The weights in the averaging operator $M_N$ are appropriately well-behaving, making $M_N$ comparable to 
\begin{equation}
A_N\mathfrak{f}(x)\coloneqq\frac{1}{N}\sum_{n=1}^Nf_n(x)\text{.}
\end{equation}
Section~$\ref{MTE}$ is dedicated to performing this decomposition and establishing that $M_N$, indeed, exhibits the same behavior as $A_N$.

A more delicate approach allows us to treat the error term in the context of the theorems proposed here. Through an instance of Calder\'on's transference principle and standard considerations, we will reduce the task of asserting that our error terms are negligible in the context of both ergodic theorems to establishing
\begin{equation}\label{U31}
\big|\mathbb{E}_{x\in[-2N,2N]^2,n\in[N]}f_0(x)f_1(x+ne_1)f_2(x+ne_2)w_N(n)\big|\le C N^{-\chi}\text{,}\end{equation}
where $e_1=(1,0)$ and $e_2=(0,1)$, and
\begin{equation}\label{U32}
\big|\mathbb{E}_{x\in[-2N^2,2N^2],n\in[N]}g_0(x)g_1(x+n^2)w_N(n)\big|\le C N^{-\chi}\text{,}
\end{equation}
for some $\chi>0$, and any $1$-bounded functions $f_0,f_1,f_2$ and $g_0,g_1$ supported in $[-2N,2N]^2$ and $[-2N^2,2N^2]$ respectively. To establish the estimates above, we begin by approximating the sawtooth function $\Phi$ appearing in the weights $w_N$ by its truncated Fourier series, which naturally decomposes our weights to $w_N=w^{\text{main}}_{N}+w^{\text{error}}_{N}$, where $w^{\text{error}}_{N}$ amounts to the contribution of the tail of the Fourier series of $\Phi$. Exploiting the uniform bounds of the tail of the aforementioned expansion, and carefully choosing an appropriate truncation parameter $M(N)$ for the Fourier series we obtain
\[
\mathbb{E}_{n\in[N]}|w^{\text{error}}_N(n)|\lesssim N^{-\chi}\text{,}\quad\text{for some positive $\chi$,}
\]
allowing us to turn our attention $w_N^{\text{main}}$. The only property of the specific patterns we utilize is the fact that both averages in $\eqref{U31}$ and $\eqref{U32}$ are bounded by $\lesssim N^{-1/2}\|w_N^{\text{main}}\|_{U^3}$ and to conclude it would suffice to prove that $\|w_N^{\text{main}}\|_{U^3}\lesssim N^{1/2-\chi}$, for some $\chi>0$. Due to technical complications, we work with dyadic variants of $w_N^{\text{main}}$, and after proving suitable $U^3$-norm estimates for these dyadic pieces, see Lemma~$\ref{errorclean}$, we can conclude. The estimates $\eqref{U31}$ and $\eqref{U32}$ are proven in section~$\ref{GNB}$, while how to use such results to establish the corresponding ergodic theorems is explained in section~$\ref{CTP}$. 

In section~$\ref{ROTHsection}$ we turn our attention to Theorem~$\ref{RothFra}$ and Theorem~$\ref{corners}$. For the former, we note that the key estimate 
\[
\big|\mathbb{E}_{x\in[-2N,2N],n\in[N]}f_0(x)f_1(x+n)f_2(x+2n)w_N(n)\big|\le C N^{-\chi}\text{,}
\]
which can be established as described above and  has essentially already appeared in \cite{NcminusNc}, allows us to conclude that if a set $A\subseteq [N]$ lacks patterns of the form $n,n+\lfloor h(m)\rfloor,n+2\lfloor h(m)\rfloor$, $n,m\in\mathbb{N}$, then
\begin{equation}
\mathbb{E}_{x,n\in[N]}1_A(x)1_A(x+n)1_A(x+2n)=O(N^{-\chi})\text{.}
\end{equation}
By standard considerations, see Theorem~$\ref{FancyRoth}$, and exploiting the state of the art bounds for Roth's theorem, see Theorem~$\ref{SOTAest}$, we are able to conclude that $|A|\lesssim N\exp(-\chi' (\log N)^{1/9})$, completing the proof. The proof of Theorem~$\ref{corners}$ is similar and we provide a quick proof-sketch in the end of the final section.
\subsection{Notation}For every $x\in\mathbb{R}$ we use the standard notation
\[
\lfloor x\rfloor=\max\{n\in\mathbb{Z}:\,n\le x\}\text{,}\quad\{x\}=x-\lfloor x\rfloor\text{,} \quad\|x\|=\min\{|x-n|:\,n\in\mathbb{Z}\}\text{,}
\]
and for every $N\in[1,\infty)$ we let 
\[
[N]\coloneqq[1,N]\cap\mathbb{Z}\text{,}\quad[\pm N]\coloneqq[-N,N]\cap\mathbb{Z}\text{,}\quad\mathbb{N}_{\ge N}\coloneqq\{n\in\mathbb{N}:\,n\ge N\}\text{.}
\]

If $A,B$ are two nonnegative quantities, we write $A\lesssim B$ or $B \gtrsim A$  to denote that there exists a positive constant $C$, possibly depending on a fixed choice of $h\in\mathcal{R}_c$, such that $A\le C B$. Whenever $A\lesssim B$ and $A\gtrsim B$, we write $A\simeq B$. In general all the implicit constants appearing may depend on such a fixed choice of $h\in\mathcal{R}_c$. We note that $h(x)$ is not defined for $x<x_0$ but we choose to abuse notation; we let $h(x)$ take arbitrary values for $x\in[1,x_0]$ and all our results remain true.

Given a measurable space $(X,\mathcal{B})$, we call a function $f\colon X\to \mathbb{C}$ $1$-bounded if $f$ is measurable and $|f|\le 1$. For every function $f\colon\mathbb{Z}\to\mathbb{C}$ and $h_1\in\mathbb{Z}$ we define the difference function $\Delta_{h_1}f(x)=f(x)\overline{f(x+h_1)}$, and for every $s\in\mathbb{N}$ and $h_1,\dotsc,h_s\in\mathbb{Z}$ we define $\Delta_{h_1,\dotsc,h_s}f(x)=\Delta_{h_1}\dots\Delta_{h_s}f(x)$. For every $s\in\mathbb{N}_{\ge2}$ and every finitely supported $f\colon\mathbb{Z}\to\mathbb{C}$ we define the (unnormalised) Gowers $U^s$-norm by
\[
\|f\|_{U^s}=\bigg(\sum_{x,h_1,\dotsc,h_s\in\mathbb{Z}}\Delta_{h_1,\dotsc,h_s}f(x)\bigg)^{1/2^s}\text{.}
\]
For every $N\in\mathbb{N}$ we define 
\[
\mu_N(n)=\frac{|\{(h_1,h_2)\in[N]:\,h_1-h_2=n\}|}{N^2}\text{,}
\]
and we note that $\mu_N(n)\lesssim N^{-1}1_{[-N,N]}(n)$. 

We denote $e^x$ by $\exp(x)$ and $e^{2\pi i x}$ by $e(x)$, and we say that a function $f$ is supported in a set $A$, if $\supp(f)\subseteq A$.
\section*{Acknowledgments}
The author would like to thank Nikos Frantzikinakis, Borys Kuca and Mariusz Mirek for several insightful discussions.
\section{Main term extraction}\label{MTE}
In this section we perform the first step of the ``change of variables'' procedure. The approach is analogous to the one in \cite{NcminusNc} and the reader is encouraged to compare the present section with section~2 from the aforementioned work. We opted for an abstract formulation here, the reader may think of $f_n$ below as $f \circ T^n\cdot g\circ S^{n}$ or $f\circ T^{n^2}$. 
\begin{proposition}\label{MaintermExtraction}
Fix $c\in[1,2)$, $h\in\mathcal{R}_c$ and let $\varphi$ be its compositional inverse. Assume $(X,\mathcal{B},\mu)$ is a probability space and  $(f_n)_{n\in\mathbb{N}}$ is a sequence of $1$-bounded functions on $X$ such that for some $f\in L^2_{\mu}(X)$  we have
\begin{equation}\label{Czconverge}
\lim_{N\to\infty}\|\mathbb{E}_{n\in[N]}f_n-f\|_{L^2_{\mu}(X)}=0\text{.}
\end{equation}
Then we have that
\begin{equation}\label{FirstReduction}
\limsup_{N\to\infty}\|\mathbb{E}_{n\in\mathbb{N}_h\cap[N]}f_n-f\|_{L^2_{\mu}(X)}\le \limsup_{N\to\infty}\Big\|\sum_{n\in[N]}\frac{\Phi(-\varphi(n+1))-\Phi(-\varphi(n))}{\lfloor \varphi(N)\rfloor}f_n\Big\|_{L^2_{\mu}(X)}\text{,}
\end{equation}
where $\Phi(x)=\{x\}-1/2$.
\end{proposition}
\begin{proof}
The identity $1_{\mathbb{N}_h}(n)=\lfloor -\varphi(n)\rfloor-\lfloor -\varphi(n+1)\rfloor$, which holds for $n\gtrsim 1$, see Lemma~2.1 in \cite{NcminusNc}, yields
\begin{multline}\label{basicfirstsplit}
\mathbb{E}_{n\in\mathbb{N}_h\cap[N]}f_n(x)=\frac{1}{|\mathbb{N}_h\cap [N]|}\sum_{n\in [N]}1_{\mathbb{N}_h}(n)f_n(x)
\\
=\frac{1}{\lfloor \varphi(N)\rfloor}\sum_{n\in [N]}\big(\varphi(n+1)-\varphi(n)\big)f_n(x)+\frac{1}{\lfloor \varphi(N)\rfloor}\sum_{n\in [N]}\big(\Phi(-\varphi(n+1))-\Phi(-\varphi(n))\big)f_n(x)+O\big(\varphi(N)^{-1}\big)\text{,}
\end{multline}
where the implied constant may depend only on $h$. Thus
\begin{multline}\label{mainredjustification}
\limsup_{N\to\infty}\|\mathbb{E}_{n\in\mathbb{N}_h\cap[N]}f_n-f\big\|_{L^2_{\mu}(X)}\le\limsup_{N\to\infty}\bigg\|\frac{1}{\lfloor \varphi(N)\rfloor}\sum_{n\in [N]}\big(\varphi(n+1)-\varphi(n)\big)f_n-f\bigg\|_{L^2_{\mu}(X)}
\\
+\limsup_{N\to\infty}\bigg\|\sum_{n\in [N]}\frac{\Phi(-\varphi(n+1))-\Phi(-\varphi(n))}{\lfloor \varphi(N)\rfloor}f_n\bigg\|_{L^2_{\mu}(X)}\text{.}
\end{multline}
To establish $\eqref{FirstReduction}$ and conclude our proof it suffices to show that
\begin{equation}
\label{SBParts}
\lim_{N\to\infty}\Big\|\frac{1}{\lfloor \varphi(N)\rfloor}\sum_{n\in [N]}\big(\varphi(n+1)-\varphi(n)\big)f_n-f\Big\|_{L^2_{\mu}(X)}=0\text{.}
\end{equation}
Firstly, note that summation by parts yields
\begin{multline}\label{twosummands}
\frac{1}{\lfloor \varphi(N)\rfloor}\sum_{n=1}^N\big(\varphi(n+1)-\varphi(n)\big)f_n
\\
=\frac{N\big(\varphi(N+1)-\varphi(N)\big)}{\lfloor\varphi(N)\rfloor}\mathbb{E}_{n\in[N]}f_n+\sum_{n=1}^{N-1}\frac{n\big((\varphi(n+1)-\varphi(n))-(\varphi(n+2)-\varphi(n+1))\big)}{\lfloor \varphi(N)\rfloor}\mathbb{E}_{m\in[n]}f_m\text{.}
\end{multline}

For the first summand using the basic properties of $\varphi$ and the Mean Value Theorem one straightforwardly obtains
\begin{equation}
\lim_{N\to\infty}\frac{N\big(\varphi(N+1)-\varphi(N)\big)}{\lfloor\varphi(N)\rfloor}=1/c\eqqcolon \gamma\text{,}\quad\text{see 2.14 in \cite{NcminusNc}, page 6,}
\end{equation} 
and by taking into account $\eqref{Czconverge}$ we get
\begin{equation}
\lim_{N\to\infty}\bigg\|\frac{N\big(\varphi(N+1)-\varphi(N)\big)}{\lfloor\varphi(N)\rfloor}\mathbb{E}_{n\in[N]}f_n-\gamma f\bigg\|_{L^2_{\mu}(X)}=0\text{.}
\end{equation}

For the second summand we will use the following simple variant of Toeplitz theorem for normed vector spaces.
\begin{proposition}\label{Toeplitznvs}
For every $N\in\mathbb{N}$ let $(c_{N,k})_{k\in[N]}$ be real numbers. Assume that the following conditions hold:
\begin{itemize}
\item[\normalfont{(i)}] For every $k\in\mathbb{N}$ we have that $\lim_{N\to\infty}c_{N,k}=0$,
\item[\normalfont{(ii)}] $\lim_{N\to\infty}\sum_{k=1}^Nc_{N,k}=1$,
\item[\normalfont{(iii)}] $\sup_{N\in\mathbb{N}}\sum_{k=1}^N|c_{N,k}|<\infty$.
\end{itemize}
Assume $(V,\|\cdot\|)$ is a complex normed vector space and $(v_n)_{n\in\mathbb{N}}$ a sequence on $V$ with $\|v_n\|\le 1$ for all $n\in\mathbb{N}$ and $\lim_{n\to\infty}\|v_n-v\|=0$ for some $v\in V$. Then we have that $\lim_{N\to\infty}\big\|\sum_{k=1}^Nc_{N,k}v_k-v\big\|=0$.
\end{proposition}
\begin{proof}
This result is a simple generalization of the the standard Toeplitz theorem; see Theorem~2.6 in \cite{NcminusNc}, as well as pages 42-48 in \cite{Toeplitz}. For the sake of completeness we provide a proof.

Let $\varepsilon>0$ and let $C\coloneqq 1+\sup_{N\in\mathbb{N}}\sum_{k=1}^N|c_{N,k}|$. By (ii) there exists $N_0=N_0(\varepsilon)\in\mathbb{N}$ such that for all $N\ge N_0$ we have
\[
\Big|\sum_{k=1}^Nc_{N,k}-1\Big|<\frac{\varepsilon}{3}\text{.}
\] 
We also have that there exists $N_1=N_1(\varepsilon,C)\in\mathbb{N}$ such that for all $k\ge N_1$ we have $\|v_k-v\|\le \frac{\varepsilon}{3C}$. By (i) we get that for every $k\in\mathbb{N}$ there exists $N_2^{(k)}=N_2^{(k)}(\varepsilon,C)\in\mathbb{N}$ such that for all $N\ge N_2^{(k)}$ we have $|c_{N,k}|\le \frac{\varepsilon}{6N_1}$. Let $N_3(\varepsilon,C)\coloneqq \max\{N_2^{(k)},\,k\in[N_1]\}$. Then, for all $N> \max\{ N_0(\varepsilon),N_1(\varepsilon,C),N_3(\varepsilon,C)\}$, we have
\begin{multline*}
\Big\|v-\sum_{k=1}^Nc_{N,k}v_k\Big\|\le \Big\|v-\sum_{k=1}^Nc_{N,k}v\Big\|+\Big\|\sum_{k=1}^Nc_{N,k}v-\sum_{k=1}^Nc_{N,k}v_k\Big\|=\Big|1-\sum_{k=1}^Nc_{N,k}\Big|\|v\|+\Big\|\sum_{k=1}^Nc_{N,k}(v-v_k)\Big\|
\\
< \frac{\varepsilon}{3}\cdot 1+\sum_{k=1}^N|c_{N,k}|\|v-v_k\|\le \frac{\varepsilon}{3}+\sum_{k=1}^{N_1}|c_{N,k}|\|v-v_k\|+\sum_{k=N_1+1}^{N}|c_{N,k}|\|v-v_k\|<\frac{\varepsilon}{3}+N_1\cdot \frac{\varepsilon}{6N_1}\cdot 2+C\cdot\frac{\varepsilon}{3C}=\varepsilon\text{.}
\end{multline*}
The proof is complete.
\end{proof} 
Returning to the second summand of $\eqref{twosummands}$, we begin with the case $c>1$. We note that the three properties for the weights
\[
c_{N,k}\coloneqq 1_{[N-1]}(k)\frac{k\big((\varphi(k+1)-\varphi(k))-(\varphi(k+2)-\varphi(k+1))\big)}{(1-\gamma)\lfloor\varphi(N)\rfloor}
\] 
can be established using the basic properties of $\varphi$. The property (i) clearly holds, and since $c_{N,k}\ge 0$ for all $k\gtrsim 1$, we have that if (ii) is true then (iii) holds as well. The second property follows from the fact that for all $c\in[1,2)$ we have that
\begin{equation}\label{unweightedlimit}
\lim_{N\to\infty}\sum_{k=1}^{N-1} \frac{k\big((\varphi(k+1)-\varphi(k))-(\varphi(k+2)-\varphi(k+1))\big)}{\lfloor\varphi(N)\rfloor}=1-\gamma\text{,}\quad\text{see 2.15 in  \cite{NcminusNc}, pages 6-7.}
\end{equation}
By $\eqref{Czconverge}$, we see that Proposition~$\ref{Toeplitznvs}$ is applicable for $v_n=\mathbb{E}_{m\in[n]}f_m$ and $v=f$, yielding
\begin{equation}
\lim_{N\to\infty}\Big\|\sum_{k=1}^{N-1}\frac{k\big((\varphi(k+1)-\varphi(k))-(\varphi(k+2)-\varphi(k+1))\big)}{(1-\gamma)\lfloor\varphi(N)\rfloor}\mathbb{E}_{m\in[k]}f_m-f\Big\|_{L^2_{\mu}(X)}=0\text{,}
\end{equation}
and thus the $L^2_{\mu}(X)$-limit of the second summand of $\eqref{twosummands}$ is $(1-\gamma)f$.

For the case $c=1$ we have
\begin{multline}
\limsup_{N\to\infty}\Big\|\sum_{n=1}^{N-1}\frac{n\big((\varphi(n+1)-\varphi(n))-(\varphi(n+2)-\varphi(n+1))\big)}{\lfloor \varphi(N)\rfloor}\mathbb{E}_{m\in[n]}f_m\Big\|_{L^2_{\mu}(X)}
\\
\le \limsup_{N\to\infty}\sum_{n=1}^{N-1}\frac{n\big|(\varphi(n+1)-\varphi(n))-(\varphi(n+2)-\varphi(n+1))\big|}{\lfloor \varphi(N)\rfloor}
\\
=\limsup_{N\to\infty}\sum_{n=1}^{N-1}\frac{n\Big(\big(\varphi(n+1)-\varphi(n)\big)-\big(\varphi(n+2)-\varphi(n+1)\big)\Big)}{\lfloor \varphi(N)\rfloor}=0\text{,}
\end{multline}
where the last inequality can be justified by noting that the expression inside the absolute value is nonnegative for $n\gtrsim 1$, and the last equality is justified by $\eqref{unweightedlimit}$.

In both cases, the $L^2_{\mu}(X)$-limit of the second summand of $\eqref{twosummands}$ is $(1-\gamma)f$. Thus we get
\begin{multline}
\limsup_{N\to\infty}\bigg\|\frac{1}{\lfloor \varphi(N)\rfloor}\sum_{n\in [N]}\big(\varphi(n+1)-\varphi(n)\big)f_n-f\bigg\|_{L^2_{\mu}(X)}
\\
\le 
\limsup_{N\to\infty}\bigg\|\frac{N\big(\varphi(N+1)-\varphi(N)\big)}{\lfloor\varphi(N)\rfloor}\mathbb{E}_{n\in[N]}f_n-\gamma f\bigg\|_{L^2_{\mu}(X)}
\\
+
\limsup_{N\to\infty}\bigg\|\sum_{n=1}^{N-1}\frac{n\Big(\big(\varphi(n+1)-\varphi(n)\big)-\big(\varphi(n+2)-\varphi(n+1)\big)\Big)}{\lfloor \varphi(N)\rfloor}\mathbb{E}_{m\in[n]}f_m-(1-\gamma) f\bigg\|_{L^2_{\mu}(X)}=0\text{.}
\end{multline}
This justifies $\eqref{SBParts}$ and completes the proof.
\end{proof}
\section{Gowers norm bounds}\label{GNB}
This section is devoted to establishing the estimates $\eqref{U31}$ and $\eqref{U32}$. We note that although it would be possible to proceed in an abstract manner and continue working with arbitrary sequences of $1$-bounded functions $\mathfrak{f}=(f_n)_{n\in\mathbb{N}} $, we decided to make things concrete here to keep the statements of our theorems reasonable. We do note however, and it will be apparent from the proof, that our treatment may address all patterns inducing averaging operators appropriately bounded by $U^3$-norms. The process will be carried out in two steps, the first of which is an instance of the fact that the operators at hand are controlled by $U^3$-norms and is quite straightforward, and the second one amounts to, loosely speaking, bounding the $U^3$-norm of the error kernel. We remind the reader that 
\[
 w_N(n)\coloneqq 1_{[N]}(n)\frac{N\big(\Phi(-\varphi(n+1))-\Phi(-\varphi(n))\big)}{\lfloor \varphi(N)\rfloor}\text{.}
\]
\begin{proposition}\label{KeyProp1}Fix $c\in[1,23/22)$ and $h\in\mathcal{R}_c$. Assume $N\in\mathbb{N}$ and $f_0,f_1,f_2\colon \mathbb{Z}^2\to\mathbb{C}$ are $1$-bounded and supported in $[\pm 2N]^2$. Then there exist positive constants $C=C(h)$ and $\chi=\chi(h)$ such that
\begin{equation}
\big|\mathbb{E}_{x\in[\pm 2N]^2,n\in[N]}f_0(x)f_1(x+ne_1)f_2(x+ne_2)w_N(n)\big|\le C N^{-\chi}\text{,}
\end{equation}
where $e_1=(1,0)$ and $e_2=(0,1)$.
\end{proposition}
\begin{proposition}\label{KeyProp2}Fix $c\in[1,23/22)$ and $h\in\mathcal{R}_c$. Assume $N\in\mathbb{N}$ and $f_0,f_1\colon\mathbb{Z}\to\mathbb{C}$ are $1$-bounded and supported in $[\pm 2N^2]$. Then there exist positive constants $C=C(h)$ and $\chi=\chi(h)$ such that
\begin{equation}\label{boundU3type}
\big|\mathbb{E}_{x\in[\pm 2N^2],n\in[N]}f_0(x)f_1(x+n^2)w_N(n)\big|\le C N^{-\chi}\text{.}
\end{equation}
\end{proposition}
We also formulate the following proposition, the proof of which is very similar to the ones of the previous two, and has essentially appeared in \cite{NcminusNc}.
\begin{proposition}\label{KeyProp3}Fix $c\in[1,23/22)$ and $h\in\mathcal{R}_c$. Assume $N\in\mathbb{N}$ and $f_0,f_1,f_2\colon\mathbb{Z}\to\mathbb{C}$ are $1$-bounded and supported in $[\pm 2N]$. Then there exist positive constants $C=C(h)$ and $\chi=\chi(h)$ such that
\begin{equation}\label{boundU3type}
\big|\mathbb{E}_{x\in[\pm 2N],n\in[N]}f_0(x)f_1(x+n)f_2(x+2n)w_N(n)\big|\le C N^{-\chi}\text{.}
\end{equation}
\end{proposition}

\subsection{Bounds by the $U^3$-norm}
\begin{lemma}\label{U3forFRA}Assume $N\in\mathbb{N}$, $f_0,f_1,f_2\colon \mathbb{Z}^2\to\mathbb{C}$ are $1$-bounded and supported in $[\pm 2N]^2$, and $f_3\colon\mathbb{Z}\to\mathbb{C}$ is $1$-bounded and supported in $[N]$. Then we have
\begin{equation}\label{U3box}
\big|\mathbb{E}_{x\in[\pm 2N]^2,n\in[N]}f_0(x)f_1(x+ne_1)f_2(x+ne_2)f_3(n)\big|\lesssim N^{-\frac{1}{2}}\|f_3\|_{U^3}\text{,}
\end{equation}
where the implied constant is absolute.
\end{lemma}
Before proceeding with the proof, let us remark that the factor $N^{-\frac{1}{2}}$ in the right hand side of the above estimate naturally appears because we have chosen to work with the unnormalised $U^3$ Gowers norm and $\|1_{[N]}\|_{U^3}\simeq N^{1/2}$.
\begin{proof}
The proof is standard but for the sake of completeness we provide the details here. Taking into account the supports we see that to establish $\eqref{U3box}$ it suffices to show that
\begin{equation}\label{goal25}
\Big|\sum_{x\in\mathbb{Z}^2}\sum_{n\in\mathbb{Z}}f_0(x)f_1(x+ne_1)f_2(x+ne_2)f_3(n)\Big|^8\lesssim N^{20}\|f_3\|^8_{U^3}\text{.}
\end{equation}
By Cauchy-Schwarz and by taking into account the support of $f_0$ and the fact that it is $1$-bounded we get 
\begin{equation}\label{GOAL}
\Big|\sum_{x\in\mathbb{Z}^2}\sum_{n\in\mathbb{Z}}f_0(x)f_1(x+ne_1)f_2(x+ne_2)f_3(n)\Big|^2\lesssim N^2\sum_{x\in\mathbb{Z}^2}\Big|\sum_{n\in[N]}f_1(x+ne_1)f_2(x+ne_2)f_3(n)\Big|^2\text{.}
\end{equation}
An application of van der Corput inequality, see for example Lemma~3.1 in \cite{Pre} or Lemma~3.1 in \cite{PelPre} bounds the right-hand side of the above expression by
\begin{multline}
\lesssim N^3\sum_{x\in\mathbb{Z}^2}\sum_{h_1\in\mathbb{Z}}\mu_N(h_1)\sum_{n\in J(h_1)} f_1(x+ne_1)\overline{f_1(x+(n+h_1)e_1)}f_2(x+ne_2)\overline{f_2(x+(n+h_1)e_2)}\Delta_{h_1}f_3(n)
\\
=N^3\sum_{h_1\in\mathbb{Z}}\mu_N(h_1)\sum_{n\in J(h_1)}\Delta_{h_1}f_3(n) \sum_{y\in\mathbb{Z}^2}f_1(y)\overline{f_1(y+h_1e_1)}f_2(y-ne_1+ne_2)\overline{f_2(y-ne_1+(n+h_1)e_2)}
\\
=N^3\sum_{y\in\mathbb{Z}^2}\sum_{h_1\in\mathbb{Z}}\mu_N(h_1)f_1(y)\overline{f_1(y+h_1e_1)}\sum_{n\in J(h_1)} f_2(y-ne_1+ne_2)\overline{f_2(y-ne_1+(n+h_1)e_2)}\Delta_{h_1}f_3(n)\text{,}
\end{multline}
where $J(h_1)\coloneqq [N]\cap([N]-h_1)$. Squaring the left-hand side of $\eqref{GOAL}$ and repeating the process yields
\begin{multline}
\Big|\sum_{x\in\mathbb{Z}^2}\sum_{n\in\mathbb{Z}}f_0(x)f_1(x+ne_1)f_2(x+ne_2)f_3(n)\Big|^4\lesssim N^6\Big(\sum_{y\in\mathbb{Z}^2}\sum_{h_1\in\mathbb{Z}}|\mu_N(h_1)f_1(y)\overline{f_1(y+h_1e_1)}|^2\Big)\cdot
\\
\cdot\Big(\sum_{y\in\mathbb{Z}^2}\sum_{h_1\in[\pm N]}\Big|\sum_{n\in J(h_1)} f_2(y-ne_1+ne_2)\overline{f_2(y-ne_1+(n+h_1)e_2)}\Delta_{h_1}f_3(n)\Big|^2\Big)
\\
\lesssim N^8\sum_{y\in\mathbb{Z}^2}\sum_{h_1\in[\pm N]}\sum_{h_2\in\mathbb{Z}}\mu_N(h_2)\sum_{n\in J(h_1,h_2)}f_2(y-ne_1+ne_2)\overline{f_2(y-ne_1+(n+h_1)e_2)}\cdot
\\
\cdot \overline{f_2(y-(n+h_2)e_1+(n+h_2)e_2)}f_2(y-(n+h_2)e_1+(n+h_1+h_2)e_2)\Delta_{h_1,h_2}f_3(n)
\\
=N^8\sum_{h_1\in[\pm N]}\sum_{h_2\in\mathbb{Z}}\mu_N(h_2)\sum_{n\in J(h_1,h_2)}\Delta_{h_1,h_2}f_3(n)\sum_{x\in\mathbb{Z}^2}f_2(x)\overline{f_2(x+h_1e_2)}\cdot
\\
\cdot \overline{f_2(x-h_2e_1+h_2e_2)}f_2(x-h_2e_1+(h_1+h_2)e_2)\text{,}
\end{multline}
where $J(h_1,h_2)=J(h_1)\cap (J(h_1)-h_2)$. Repeating one final time we obtain
\begin{multline}
\Big|\sum_{x\in\mathbb{Z}^2}\sum_{n\in\mathbb{Z}}f_0(x)f_1(x+ne_1)f_2(x+ne_2)f_3(n)\Big|^8
\\
\lesssim N^{16}\sum_{h_1\in[\pm N]}\sum_{h_2\in\mathbb{Z}}|\mu_N(h_2)|^2\Big|\sum_{x\in\mathbb{Z}^2}f_2(x)\overline{f_2(x+h_1e_2)}\overline{f_2(x-h_2e_1+h_2e_2)}f_2(x-h_2e_1+(h_1+h_2)e_2)\Big|^2\cdot
\\
\cdot\sum_{h_1,h_2\in[\pm N]}\Big|\sum_{n\in J(h_1,h_2)}\Delta_{h_1,h_2}f_3(n)\Big|^2
\\
\lesssim N^{21}\sum_{h_1,h_2\in[\pm N]}\sum_{h_3\in\mathbb{Z}}\mu_N(h_3)\sum_{n\in J(h_1,h_2,h_3)}\Delta_{h_1,h_2,h_3}f_3(n)=N^{21}\sum_{n,h_1,h_2,h_3\in\mathbb{Z}}\mu_N(h_3)\Delta_{h_1,h_2,h_3}f_3(n)\text{,}
\end{multline}
where $J(h_1,h_2,h_3)=J(h_1,h_2)\cap(J(h_1,h_2)-h_3)$, and the last equality can be justified by taking into account the support of $f_3$. Finally, we note that since the $U^2$-norm is positive we may estimate as follows
\begin{multline}
N^{21}\Big|\sum_{n,h_1,h_2,h_3\in\mathbb{Z}}\mu_N(h_3)\Delta_{h_1,h_2,h_3}f_3(n)\Big|\le N^{21}\sum_{h_3\in\mathbb{Z}}|\mu_N(h_3)|\Big|\sum_{n,h_1,h_2\in\mathbb{Z}}\Delta_{h_1,h_2,h_3}f_3(n)\Big|
\\
\lesssim N^{20}\sum_{h_3\in\mathbb{Z}}\|\Delta_{h_3}f_3\|^4_{U^2}=N^{20}\|f_3\|_{U^3}^{8}\text{,}
\end{multline}
and thus $\eqref{goal25}$ is justified and the proof is complete.
\end{proof}
\begin{lemma}\label{U3controlforsquares}Assume $N\in\mathbb{N}$, $f_0,f_1\colon\mathbb{Z}\to\mathbb{C}$ are $1$-bounded and supported in $[\pm 2N^2]$, and $f_2\colon\mathbb{Z}\to\mathbb{C}$ is $1$-bounded and supported in $[N]$. Then we have
\begin{equation}\label{boundU3type}
\big|\mathbb{E}_{x\in[\pm 2N^2],n\in[N]}f_0(x)f_1(x+n^2)f_2(n)\big|\lesssim N^{-\frac{1}{2}}\|f_2\|_{U^3}\text{,}
\end{equation}
where the implied constant is absolute.
\end{lemma}
\begin{proof}
The proof is similar to the proof of Lemma~$\ref{U3forFRA}$. It suffices to prove that
\begin{equation}
\Big|\sum_{x\in[\pm 2N^2]}\sum_{n\in[N]}f_0(x)f_1(x+n^2)f_2(n)\Big|^8\lesssim N^{20}\|f_2\|^8_{U^3}\text{.}
\end{equation}
Letting $J(h_1)$, $J(h_1,h_2)$,$J(h_1,h_2,h_3)$ be as in the previous proof and following the same steps we obtain
\begin{multline}
\Big|\sum_{x\in\mathbb{Z}}\sum_{n\in\mathbb{Z}}f_0(x)f_1(x+n^2)f_2(n)\Big|^2
\lesssim N^2\sum_{x\in\mathbb{Z}}\Big|\sum_{n\in\mathbb{Z}}f_1(x+n^2)f_2(n)\Big|^2
\\
\lesssim 
N^3\sum_{x\in\mathbb{Z}}\sum_{h_1\in\mathbb{Z}}\mu_N(h_1)\sum_{n\in J(h_1)}f_1(x+n^2)\overline{f_1(x+n^2+2nh_1+h_1^2)}\Delta_{h_1}f_2(n)
\\
=N^3\sum_{h_1\in\mathbb{Z}}\mu_N(h_1)\sum_{n\in J(h_1)}\sum_{y\in\mathbb{Z}}f_1(y)\overline{f_1(y+2nh_1+h_1^2)}\Delta_{h_1}f_2(n)\text{,}
\end{multline}
and thus
\begin{multline}
\Big|\sum_{x\in\mathbb{Z}}\sum_{n\in\mathbb{Z}}f_0(x)f_1(x+n^2)f_2(n)\Big|^4
\\
\lesssim N^6\Big(\sum_{h_1\in\mathbb{Z}}\sum_{y\in\mathbb{Z}}|f_2(y)\mu_N(h_1)|^2\Big) \Big(\sum_{h_1\in[\pm N]}\sum_{y\in\mathbb{Z}}\Big|\sum_{n\in J(h_1)}\overline{f_1(y+2nh_1+h_1^2)}\Delta_{h_1}f_2(n)\Big|^2\Big)
\\
\lesssim N^8\sum_{h_1\in[\pm N]}\sum_{y\in\mathbb{Z}}\sum_{h_2\in\mathbb{Z}}\mu_N(h_2)\sum_{n\in J(h_1,h_2)}\overline{f_1(y+2nh_1+h_1^2)}f_1(y+2(n+h_2)h_1+h_1^2)\Delta_{h_1,h_2}f_2(n)
\\
=N^8\sum_{h_1\in[\pm N]}\sum_{h_2\in\mathbb{Z}}\mu_N(h_2)\sum_{n\in J(h_1,h_2)}\sum_{x\in\mathbb{Z}}\overline{f_1(x)}f_1(x+2h_2h_1)\Delta_{h_1,h_2}f_2(n)\text{.}
\end{multline}
One final iteration yields
\begin{multline}
\Big|\sum_{x\in\mathbb{Z}}\sum_{n\in\mathbb{Z}}f_0(x)f_1(x+n^2)f_2(n)\Big|^8
\\
\lesssim N^{16}\Big(\sum_{h_1\in[\pm N]}\sum_{h_2\in\mathbb{Z}}|\mu_N(h_2)|^2\Big|\sum_{x\in\mathbb{Z}}\overline{f_1(x)}f_1(x+2h_2h_1)\Big|^2\Big)\Big(\sum_{h_1\in[\pm N]}\sum_{h_2\in[\pm N]}\Big|\sum_{n\in J(h_1,h_2)}\Delta_{h_1,h_2}f_2(n)\Big|^2\Big)
\\
\lesssim N^{21}\sum_{h_1,h_2\in[\pm N]}\sum_{h_3\in\mathbb{Z}}\mu_{N}(h_3)\sum_{n\in J(h_1,h_2,h_3)}\Delta_{h_1,h_2,h_3}f_2(n)\lesssim N^{20}\|f_3\|_{U^3}^8\text{,}
\end{multline}
where the last estimate can be established as in the previous proof.
\end{proof}
\subsection{Bounds for the $U^3$-norm of the error kernel}
Before applying Lemma~$\ref{U3forFRA}$ and Lemma~$\ref{U3controlforsquares}$ for estimating $\eqref{U31}$ and $\eqref{U32}$ respectively, it will be convenient to pass to a refinement of the error kernel  $w_N$. More precisely, for every $M\in\mathbb{N}_{\ge 2}$ we have 
\begin{equation}\label{Fourier}
\Phi(x)\coloneqq\{x\}-1/2=\sum_{0<|m|\le M}\frac{1}{2\pi im}e(-mx)+g_M(x)\text{,}
\end{equation}
with $g_M(x)=O\Big(\min\Big\{1,\frac{1}{M\|x\|}\Big\}\Big)$ and the implied constant is absolute, see Lemma~2.3 in \cite{NcminusNc}. We apply this for carefully chosen parameters to replace $\Phi$ in our error kernel $w_{N}$. We fix
\begin{equation}\label{choicesforsme}
\varepsilon_0\coloneqq \frac{23-22c}{40c}\text{,}\quad\sigma_0\coloneqq1-\frac{1}{c}+\varepsilon_0\text{,}\quad M\coloneqq \lfloor N^{\sigma_0}\rfloor\text{,}
\end{equation}
and by applying $\eqref{Fourier}$ we obtain
\begin{multline}
w_{N}(n)=\frac{N\big(\Phi(-\varphi(n+1))-\Phi(-\varphi(n))\big)}{\lfloor \varphi(N)\rfloor}1_{[N]}(n)
\\
=\frac{N}{\lfloor \varphi(N)\rfloor}\bigg(\sum_{0<|m|\le M}\frac{e(m\varphi(n+1))-e(m\varphi(n))}{2\pi i m}\bigg)1_{[N]}(n)
\\
+\frac{N}{\lfloor \varphi(N)\rfloor}\big(g_M(-\varphi(n+1))-g_M(-\varphi(n))\big)1_{[N]}(n)\eqqcolon w^{\text{main}}_{N}(n)+w^{\text{error}}_{N}(n) \text{.}
\end{multline}
To complete the proof of Proposition~$\ref{KeyProp1}$ and Proposition~$\ref{KeyProp2}$ it suffices to prove the following.
\begin{lemma}\label{errorclean}
Assume $c\in[1,23/22)$ and $h\in\mathcal{R}_c$. Then there exist positive constants $C=C(h)$ and $\chi=\chi(h)$ such that
\begin{equation}
\sum_{l\in\mathbb{N}_0}\|w^{\textnormal{main}}_{N,l}\|_{U^3}\le CN^{\frac{1}{2}-\chi}\quad\text{and}\quad\mathbb{E}_{n\in[N]}|w^{\textnormal{error}}_{N}(n)|\le CN^{-\chi}\text{,}
\end{equation}
where $w^{\textnormal{main}}_{N,l}(n)\coloneqq w^{\textnormal{main}}_{N}(n)1_{[2^l\min(2^{l+1},N+1))}(n)$.
\end{lemma} 
Before giving a proof we show how one may yield the desired results using the lemma above.
\begin{proof}[Concluding the proof of Proposition~$\ref{KeyProp1}$ and $\ref{KeyProp2}$]
We note that
\begin{multline*}
\big|\mathbb{E}_{x\in[\pm 2N]^2,n\in[N]}f_0(x)f_1(x+ne_1)f_2(x+ne_2)w_N(n)\big|\\
\le \big|\mathbb{E}_{x\in[\pm 2N]^2,n\in[N]}f_0(x)f_1(x+ne_1)f_2(x+ne_2)w^{
\text{main}}_N(n)\big|+\mathbb{E}_{n\in[N]}|w^{\text{error}}_N(n)|
\\
\lesssim\sum_{l\in\mathbb{N}_0}\big|\mathbb{E}_{x\in[\pm 2N]^2,n\in[N]}f_0(x)f_1(x+ne_1)f_2(x+ne_2)w^{
\text{main}}_{N,l}(n)\big|+N^{-\chi}
\\
\lesssim N^{-\frac{1}{2}}\sum_{l\in\mathbb{N}_0}\|w^{\text{main}}_{N,l}\|_{U^3}+N^{-\chi}\lesssim N^{-\chi}\text{.}
\end{multline*}
Proposition~$\ref{KeyProp2}$ can be established in an identical manner.
\end{proof}
\begin{proof}[Proof of Lemma~$\ref{errorclean}$]
The proof is identical to the proof of Proposition~2.24(ii) and Lemma~3.8 in \cite{NcminusNc}. Nevertheless, for the convenience of the reader we provide some details here. Let us begin with the second assertion. Note that 
\begin{multline}
\mathbb{E}_{n\in[N]}|w^{\textnormal{error}}_{N}(n)|\lesssim \frac{1}{\lfloor \varphi(N)\rfloor}\sum_{n\in [N]}\big(|g_M(-\varphi(n+1))|+|g_M(-\varphi(n))|\big)
\\
\lesssim \frac{1}{\varphi(N)}\sum_{n\in[N]}\bigg(\min\bigg\{1,\frac{1}{M\|\varphi(n+1)\|}\bigg\}+\min\bigg\{1,\frac{1}{M\|\varphi(n)\|}\bigg\}\bigg)
\\
\lesssim\frac{N\log M}{M\varphi(N)}+\frac{NM^{1/2}\log N}{\varphi(N)^{3/2}\sigma(N)^{1/2}}
\text{,}
\end{multline}
where we have used Lemma~2.26 in \cite{NcminusNc}. Let us mention that if $c>1$, we let $\sigma$ be the constant function $1$, and if $c=1$, $\sigma$ is as in Lemma~2.14 in \cite{MMR}. To obtain a bound of the form $\lesssim N^{-\chi}$ for some $\chi>0$, for the first summand, it suffices to have that $1-\sigma_0-\frac{1}{c}<0\iff \sigma_0>1-\frac{1}{c}$,
and for the second summand, it suffices to have $1+\frac{\sigma_0}{2}-\frac{3}{2c}<0\iff \sigma_0<\frac{3}{c}-2$, where for the case $c=1$ we took into account that $\sigma(x)^{-1}\lesssim_{\delta}x^{\delta}$ for all $\delta>0$, see Lemma~2.14 in \cite{MMR}. Any $\sigma_0\in\big(1-\frac{1}{c},\frac{3}{c}-2\big)$ is admissible, and it is not difficult to check that our choice
\[
\sigma_0\coloneqq1-\frac{1}{c}+\varepsilon_0\text{,}\quad\text{see $\eqref{choicesforsme}$,}
\]
belongs in this interval. Thus we have shown that there exists $\chi>0$ such that $\mathbb{E}_{n\in[N]}|w^{\textnormal{error}}_{N}(n)|\lesssim N^{-\chi}$, and the second assertion of the lemma is established.

For the first bound, it will be more convenient to work with an unweighted variant of the kernel. To this end, we define
\begin{equation}\label{LNdef}
L_{N,l}(n)\coloneqq \frac{\lfloor\varphi(N)\rfloor}{N}w^{\text{main}}_{N,l}(n)=1_{[2^l,\min(2^{l+1},N+1))}(n)\sum_{0<|m|\le M}\frac{e(m\varphi(n+1))-e(m\varphi(n))}{2\pi i m}\text{.}
\end{equation}
By $\eqref{Fourier}$ we get that $|L_{N,l}|\lesssim 1$, and thus for every $h_3\in\mathbb{Z}$ we have
\begin{equation}\label{UnWeighted}
\sum_{x,h_1,h_2\in\mathbb{Z}}\Delta_{h_1,h_2}\big(\Delta_{h_3}w^{\text{main}}_{N,l}\big)(x)=\|\Delta_{h_3}w^{\text{main}}_{N,l}\|_{U^2}^4=N^8\lfloor \varphi(N)\rfloor^{-8}\|\Delta_{h_3}L_{N,l}\|_{U^2}^4\text{.}
\end{equation}
Taking into account the fact that $|L_{N,l}|\lesssim 1$ and applying the inverse theorem for the $U^2$-norm, see Lemma~A.1 in \cite{PelPre}, we obtain the following bound
\begin{equation}\label{U2Inv}
\|\Delta_{h_3}L_{N,l}\|_{U^2}^4\lesssim N \sup_{\xi\in\mathbb{T}}\Big|\sum_{x\in\mathbb{Z}}\big(\Delta_{h_3}L_{N,l}(x)\big)e(x\xi)\Big|^2\text{.}
\end{equation}
We apply Lemma~3.11 from \cite{NcminusNc} with $\kappa=\frac{9c-6}{5}$ to obtain that for every $N\in\mathbb{N}$, $l\in[0,\log_2(N+1)]\cap \mathbb{Z}$ and $|h_3|\ge \varphi(2^l)^{\kappa}$ the following bound holds
\begin{equation}\label{HardEstimates}
\Big\|\sum_{x\in\mathbb{Z}}\big(\Delta_{h_3}L_{N,l}(x)\big)e(x\xi)\Big\|_{L_{d\xi}^{\infty}(\mathbb{T})}\lesssim 2^{-\frac{2l}{3}}\sigma(2^l)^{-\frac{1}{3}}\varphi(2^l)^{\frac{5-\kappa}{3}}M^{\frac{8}{3}}\text{.}
\end{equation}
Using $\eqref{UnWeighted}$, $\eqref{U2Inv}$ and $\eqref{HardEstimates}$ we may estimate as follows
\begin{multline}
\|w^{\text{main}}_{N,l}\|_{U^3}^8=\sum_{h_3\in\mathbb{Z}}\sum_{x,h_1,h_2\in\mathbb{Z}}\Delta_{h_1,h_2,h_3}w^{\text{main}}_{N,l}(x)=N^8\lfloor \varphi(N)\rfloor^{-8}\sum_{h_3\in[\pm N]}\|\Delta_{h_3}L_{N,l}\|_{U^2}^4
\\
\lesssim N^9\lfloor \varphi(N)\rfloor^{-8}\sum_{|h_3|\le \varphi(2^l)^{\kappa}}\Big\|\sum_{x\in\mathbb{Z}}\big(\Delta_{h_3}L_{N,l}(x)\big)e(x\xi)\Big\|^2_{L_{d\xi}^{\infty}(\mathbb{T})}\\+N^9\lfloor \varphi(N)\rfloor^{-8}\sum_{\varphi(2^l)^{\kappa}<|h_3|\le N}\Big\|\sum_{x\in\mathbb{Z}}\big(\Delta_{h_3}L_{N,l}(x)\big)e(x\xi)\Big\|^2_{L_{d\xi}^{\infty}(\mathbb{T})}
\\ \lesssim N^{11}\varphi(N)^{-8+\kappa}+N^{10}\varphi(N)^{-8}2^{-\frac{4l}{3}}\sigma(2^l)^{-\frac{2}{3}}\varphi(2^l)^{\frac{10-2\kappa}{3}}M^{\frac{16}{3}}
\\
=N^{11}\varphi(N)^{-8+\frac{9c-6}{5}}+N^{10}\varphi(N)^{-8}2^{-\frac{4l}{3}}\sigma(2^l)^{-\frac{2}{3}}\varphi(2^l)^{\frac{10-2\kappa}{3}}N^{\frac{16}{3}-\frac{16}{3c}+\frac{16}{3}\varepsilon_0}
\\
=N^{11}\varphi(N)^{\frac{9c-46}{5}}+N^{\frac{46}{3}-\frac{16}{3c}+\frac{16}{3}\varepsilon_0}\varphi(N)^{-8}2^{-\frac{4l}{3}}\sigma(2^l)^{-\frac{2}{3}}\varphi(2^l)^{\frac{10-2\kappa}{3}}
\text{,}
\end{multline}
and thus
\begin{equation}
\|w^{\text{main}}_{N,l}\|_{U^3}\lesssim N^{\frac{11}{8}}\varphi(N)^{\frac{9c-46}{40}}+N^{\frac{23}{12}-\frac{8}{12c}+\frac{8}{12}\varepsilon_0}\varphi(N)^{-1}2^{-\frac{l}{6}}\sigma(2^l)^{-\frac{1}{12}}\varphi(2^l)^{\frac{5-\kappa}{12}}
\text{,} 
\end{equation}
and finally
\begin{multline}
\sum_{l\in\mathbb{N}_0}
\|w^{\text{main}}_{N,l}\|_{U^3}=\sum_{0\le l\le \log_2(N+1)}
\|w^{\text{main}}_{N,l}\|_{U^3}\lesssim \log(N) N^{\frac{11}{8}}\varphi(N)^{\frac{9c-46}{40}}\\+N^{\frac{23}{12}-\frac{8}{12c}+\frac{8}{12}\varepsilon_0}\varphi(N)^{-1}\sum_{0\le l\le \log_2(N+1)}2^{-\frac{l}{6}}\sigma(2^l)^{-\frac{1}{12}}\varphi(2^l)^{\frac{5-\kappa}{12}}\lesssim \log(N) N^{\frac{11}{8}}\varphi(N)^{\frac{9c-46}{40}}\\+N^{\frac{23}{12}-\frac{8}{12c}+\frac{8}{12}\varepsilon_0}\varphi(N)^{-1}N^{-\frac{1}{6}+\frac{5}{12c}-\frac{\kappa}{12c}+\frac{1}{6}\varepsilon_0}\lesssim N^{\frac{4}{8}}N^{\frac{7}{8}}N^{\frac{9c-46}{40c}}N^{\varepsilon_0}+N^{\frac{1}{2}}N^{\frac{22}{20}-\frac{46}{40c}}N^{\varepsilon_0}\\
\lesssim N^{\frac{1}{2}}N^{\frac{44c-46}{40c}}N^{\varepsilon_0}=N^{\frac{1}{2}}N^{\frac{22c-23}{40c}}
\text{,} 
\end{multline}
where the estimation of the sum over $l$ relies on the basic properties of $\varphi$ and $\sigma$, see 3.20 in \cite{NcminusNc} for detailed calculations. The proof is complete since $(22c-23)/40c<0$ because $c<23/22$.
\end{proof}

\section{Calder\'on's Transference Principle and the proofs of the Ergodic Theorems}\label{CTP}
This section is devoted to explaining how Proposition~$\ref{KeyProp1}$ and Proposition~$\ref{KeyProp2}$ allows us to establish Theorem~$\ref{FRA}$ and Theorem~$\ref{n2pt}$ respectively. A straightforward instance of Calder\'on's transference principle,  which in this case amounts to a rather simple averaging argument, transforms $\eqref{U31}$ and $\eqref{U32}$ to $L^1$- estimates for the error terms, see Proposition~$\ref{L1controlEN}$. In the end of the section we combine these estimates with an application of Proposition~$\ref{MaintermExtraction}$ and conclude the proof of both ergodic theorems. 

We remind the reader that for every probability space $(X,\mathcal{B},\mu)$ and every sequence $(f_n)_{n\in\mathbb{N}}$ of $1$-bounded functions on $X$ we define 
\begin{equation}\label{defen}
E_N\big((f_n)_{n\in\mathbb{N}}\big)(x)\coloneqq \mathbb{E}_{n\in[N]}w_N(n)f_n(x)=\sum_{n\in[N]}\frac{\Phi(-\varphi(n+1))-\Phi(-\varphi(n))}{\lfloor \varphi(N)\rfloor}f_n(x)\text{.}
\end{equation}
\begin{proposition}\label{L1controlEN}Assume $c\in[1,23/22)$ and $h\in\mathcal{R}_c$. Let $(X,\mathcal{B},\mu)$ be a probability space and $T,S$ measure-preserving transformations on $X$ which commute. Then there exist positive constants $C=C(h)$ and $\chi=\chi(h)$ such that for every pair of $1$-bounded functions $f,g\colon X\to\mathbb{C}$ we have
\begin{equation}\label{L1saving}
\big\|E_N\big((f\circ T^n \cdot g\circ S^n)_{n\in\mathbb{N}}\big)\big\|_{L^1_{\mu}(X)}\le CN^{-\chi}\quad\text{and}\quad\big\|E_N\big((f\circ T^{n^2})_{n\in\mathbb{N}}\big)\big\|_{L^1_{\mu}(X)}\le CN^{-\chi}\text{.}
\end{equation}
\end{proposition}
\begin{proof}
By duality, $\mu$-invariance and the fact that $T,S$ commute, there exists a 1-bounded function $l$ such that
\begin{multline*}
\big\|E_N\big((f\circ T^n \cdot g\circ S^n)_{n\in\mathbb{N}}\big)\big\|_{L^1_{\mu}(X)}=\int_Xl(x)\mathbb{E}_{n\in[N]}w_N(n)f(T^nx)g(S^{n}x)d\mu(x)
\\
=\frac{1}{N^2}\sum_{(m_1,m_2)\in[N]^2}\int_Xl(T^{m_1}S^{m_2}x)\mathbb{E}_{n\in[N]}w_N(n)f(T^{m_1+n}S^{m_2}x)g(T^{m_1}S^{m_2+n}x)d\mu(x)
\\
\lesssim\int_X\big|\mathbb{E}_{(m_1,m_2)\in[\pm 2N]^2}\mathbb{E}_{n\in[N]}\big(l(T^{m_1}S^{m_2}x)1_{[N]^2}\big((m_1,m_2)\big)\big)\big(f(T^{m_1+n}S^{m_2}x)1_{[\pm 2N]^2}\big((m_1+n,m_2)\big)\big)\cdot
\\
\cdot\big(g(T^{m_1}S^{m_2+n}x)1_{[\pm 2N]^2}\big((m_1,m_2+n)\big)\big)\big)w_N(n)\big|d\mu(x)\lesssim N^{-\chi}\text{,}
\end{multline*}
where for the last step we note that for every fixed $x\in X$, we have applied Proposition~$\ref{KeyProp1}$ to the obvious functions. The second estimate is established in similar manner. There exists a 1-bounded function $l$ such that
\begin{multline*}
\|E_N\big((f\circ T^{n^2})_{n\in\mathbb{N}}\big)\|_{L^1_{\mu}(X)}=\int_Xl(x)\mathbb{E}_{n\in[N]}w_N(n)f(T^{n^2}x)d\mu(x)
\\
=\frac{1}{N^2}\sum_{m\in[N^2]}\int_Xl(T^mx)\mathbb{E}_{n\in[N]}w_N(n)f(T^{n^2+m}x)d\mu(x)
\\
\le\int_X\big|\mathbb{E}_{m\in[\pm 2N^2]}\mathbb{E}_{n\in[N]}\big(l(T^mx)1_{[N^2]}(m)\big)\big(f(T^{n^2+m}x)1_{[\pm 2N^2]}(n^2+m)\big)w_N(n)d\mu(x)\lesssim N^{-\chi}\text{,}
\end{multline*}
where we have used Proposition~$\ref{KeyProp2}$.
\end{proof}
\begin{proof}[Concluding the proof of Theorem~$\ref{FRA}$]
Fix $c\in[1,23/22)$ and $h\in\mathcal{R}_c$. Let $f,g\in L^{\infty}_{\mu}(X)$ and apply Proposition~$\ref{MaintermExtraction}$ with $f_n=f\circ T^n
\cdot g\circ S^n$ to obtain
\[
\limsup_{N\to\infty}\|\mathbb{E}_{n\in\mathbb{N}_h\cap[N]}f\circ T^n\cdot g\circ S^n-L\|_{L^2_{\mu}(X)}\le \limsup_{N\to\infty}\Big\|\sum_{n\in[N]}\frac{\Phi(-\varphi(n+1))-\Phi(-\varphi(n))}{\lfloor \varphi(N)\rfloor}f\circ T^n\cdot g\circ S^n\Big\|_{L^2_{\mu}(X)}\text{,}
\]
where $L$ is the $L^2_{\mu}(X)$-limit of $\frac{1}{N}\sum_{n=1}^Nf\circ T^n \cdot g\circ S^n$, which is known to exist, see for example \cite{TAO}. Taking into account $\eqref{defen}$, it suffices to prove that $\lim_{N\to\infty}\big\|E_N\big((f\circ T^n \cdot g\circ S^n)_{n\in\mathbb{N}}\big)\big\|_{L^2_{\mu}(X)}=0$. By Proposition $\eqref{L1controlEN}$ we obtain
\begin{equation}\label{L2IM}
\big\|E_N\big((f\circ T^n \cdot g\circ S^n)_{n\in\mathbb{N}}\big)\big\|_{L^2_{\mu}(X)}^2\le \big\|E_N\big((f\circ T^n \cdot g\circ S^n)_{n\in\mathbb{N}}\big)\big\|_{L^1_{\mu}(X)}\big\|E_N\big((f\circ T^n \cdot g\circ S^n)_{n\in\mathbb{N}}\big)\big\|_{L^\infty_{\mu}(X)}\lesssim N^{-\chi}\text{.}
\end{equation}
Here we have used the fact that $\big\|E_N\big((f_n)_{n\in\mathbb{N}}\big)\big\|_{L^\infty_{\mu}(X)}\lesssim 1$ holds for every sequence of $1$-bounded functions. To see this, note that the decomposition from $\eqref{basicfirstsplit}$ yields
\begin{multline}
E_N\big((f_n)_{n\in\mathbb{N}}\big)(x)=\frac{1}{\lfloor \varphi(N)\rfloor}\sum_{n\in [N]}\big(\Phi(-\varphi(n+1))-\Phi(-\varphi(n))\big)f_n(x)
\\
=
\mathbb{E}_{n\in\mathbb{N}_h\cap[N]}f_n(x)-
\frac{1}{\lfloor \varphi(N)\rfloor}\sum_{n\in [N]}\big(\varphi(n+1)-\varphi(n)\big)f_n(x)+O\big(\varphi(N)^{-1}\big)=O(1)\text{.}
\end{multline}
Thus by $\eqref{L2IM}$ we immediately obtain $\lim_{N\to\infty}\big\|E_N\big((f\circ T^n \cdot g\circ S^n)_{n\in\mathbb{N}}\big)\big\|_{L^2_{\mu}(X)}=0$ which yields the desired result.
\end{proof}
\begin{proof}[Concluding the proof of Theorem~$\ref{n2pt}$]
The proof is very similar, essentially using the second assertion of Proposition~$\ref{L1controlEN}$ in an identical manner to the one from section~4 in \cite{NcminusNc}. For the sake of completeness we provide the details here. Let $(X,\mathcal{B},\mu)$ be a probability space and $f\in L^{\infty}_{\mu}(X)$, and fix $\lambda\in(1,2]$. Then there exists $g\in L^{\infty}_{\mu}(X)$ such that
\[
\lim_{N\to\infty}\mathbb{E}_{n\in[N]}f(T^{n^2}x)=g(x)\quad\text{for $\mu$-a.e. $x\in X$, see \cite{bg2}.}
\] 
Let $X_0\in\mathcal{B}$ be of measure $1$ with its elements satisfying the above assertion, and fix $x_0\in X_0$. It is easy to see that the argument in the proof of Proposition~$\ref{MaintermExtraction}$ for the trivial probability space on $\{x_0\}$ and for $f_n(x_0)=f(T^{n^2}(x_0))$ can yield
\begin{multline}
\limsup_{k\to\infty}|\mathbb{E}_{n\in\mathbb{N}_h\cap[\lambda^k]}f(T^{n^2}x_0)-g(x_0)|
\\
\le \limsup_{k\to\infty}\Big|\sum_{n\in[\lambda^k]}\frac{\Phi(-\varphi(n+1))-\Phi(-\varphi(n))}{\lfloor \varphi(N)\rfloor}f(T^{n^2}x_0)\Big|=\limsup_{k\to\infty}\big|E_{\lfloor \lambda^k\rfloor}\big((f\circ T^{n^2})_{n\in\mathbb{N}}\big)(x_0)\big|\text{.}
\end{multline}
By Proposition~$\ref{L1controlEN}$ we obtain 
\[
\Big\|\sum_{k\in\mathbb{N}_0}\big|E_{\lfloor \lambda^k\rfloor}\big((f\circ T^{n^2})_{n\in\mathbb{N}}\big)\big|\Big\|_{L^1_{\mu}(X)}\le\sum_{k\in\mathbb{N}_0}\|E_{\lfloor \lambda^k\rfloor}\big((f\circ T^{n^2})_{n\in\mathbb{N}}\big)\|_{L^1_{\mu}(X)}\lesssim \sum_{k\in\mathbb{N}_0}\lambda^{-\chi k}<\infty\text{,}
\]
which immediately implies
\begin{equation}\label{ptgoalfinal}
\lim_{k\to\infty} E_{\lfloor \lambda^k\rfloor}\big((f\circ T^{n^2})_{n\in\mathbb{N}}\big)(x)=0\quad\text{for $\mu$-a.e. $x\in X$}\text{.}
\end{equation}
Thus
\[
\limsup_{k\to\infty}|\mathbb{E}_{n\in\mathbb{N}_h\cap[\lambda^k]}f(T^{n^2}x)-g(x)|=0\quad\text{for $\mu$-a.e. $x\in X$}\text{,}
\]
which, since $\lambda\in(1,2]$ was arbitrary, implies the desired result.
\end{proof}

\section{Roth's theorem and Corners for fractional powers}\label{ROTHsection}
This section is devoted to the proofs of Theorem~$\ref{RothFra}$, Corollary~$\ref{RothFraPrimes}$ and Theorem~$\ref{corners}$. We begin with establishing the Roth-type estimate \eqref{RothBound}, and  following the ideas developed earlier we pass from the count of configurations of the form $n,n+\lfloor h(m)\rfloor,n+2\lfloor h(m)\rfloor$ to the conventional ones $n,n+k,n+2k$. Our proof relies on the decomposition $\eqref{basicfirstsplit}$, Proposition~$\ref{KeyProp3}$ and the following two theorems. Before stating them, let us introduce some notation.

We let $r_3(N)$ denote the maximal size of a subset of $[N]$ without non-trivial three-term arithmetic progressions and we let $\alpha_3(N)$ be the corresponding density, i.e.:
\begin{equation}\label{RothNotation}\alpha_3(N)\coloneqq\frac{r_3(N)}{N}\text{.}
\end{equation}  
Finally, for every $N\in\mathbb{N}$ we let $\mathbb{Z}_N\coloneqq\mathbb{Z}/N\mathbb{Z}$. 
\begin{theorem}\label{FancyRoth}
There exists a positive constant $C_0$ such that for every $N\in\mathbb{N}$, $\alpha>0$ and $f\colon\mathbb{Z}_N\to [0,1]$ with $\mathbb{E}_{x\in\mathbb{Z}_N}f(x)\ge \alpha$, we have 
\begin{equation}
\mathbb{E}_{x,d\in \mathbb{Z}_N}[f(x)f(x+d)f(x+2d)]\ge C_0 \big(\alpha_3^{-1}(\alpha/2)\big)^{-2}\text{,}
\end{equation}
where $\alpha_3^{-1}$ denotes the inverse of the function defined in $\eqref{RothNotation}$.
\end{theorem}
\begin{proof}
See Theorem~2.1 in \cite{BorysRef}.
\end{proof}
\begin{theorem}\label{SOTAest}
There exists a positive constants $C_1$ and $\tau$ such that for all $N\in\mathbb{N}$ we have
\begin{equation}\label{stdrRoth}
\alpha_3(N)\le C_1\exp(-\tau(\log N)^{1/9})\text{.}
\end{equation}
\end{theorem}
\begin{proof}
The quasipolynomial decay first appeared in \cite{KelleyMeka}. For an exposition as well as for the slightly sharper result appearing above, see \cite{BLoomExp} and \cite{BloomSharp} respectively.
\end{proof}
\begin{proof}[Proof of Theorem~$\ref{RothFra}$]
Let $N\in\mathbb{N}$, $A\subseteq [N]$, and assume $A$ lacks patterns of the form $n,n+\lfloor h(m)\rfloor,n+2\lfloor h(m)\rfloor$, $n,m\in\mathbb{N}$. Then we have
\begin{equation}\label{zeropatternsconvenient}
\mathbb{E}_{x\in[N],n\in\mathbb{N}_h\cap[N]}1_A(x)1_A(x+n)1_A(x+2n)=0\text{.}
\end{equation}
Letting $f^{A}_n(x)=1_A(x)1_A(x+n)1_A(x+2n)$ and arguing as in the beginning of  Proposition~$\ref{MaintermExtraction}$, see $\eqref{twosummands}$, we get that for every $x\in\mathbb{Z}$ 
\begin{multline}
\sum_{n\in\mathbb{N}_h\cap[N]}f^{A}_n(x)=\sum_{n\in[N]}1_{\mathbb{N}_h}(n)f^{A}_n(x)=\sum_{n\in[N]}\big(\varphi(n+1)-\varphi(n)\big)f^{A}_n(x)
\\
+\sum_{n\in[N]}\big(\Phi(-\varphi(n+1))-\Phi(-\varphi(n))\big)f^{A}_n(x)+O(1)\text{,}
\end{multline}
and
\begin{multline}
\frac{1}{\lfloor \varphi(N)\rfloor}\sum_{n=1}^N\big(\varphi(n+1)-\varphi(n)\big)f^A_n(x)=\frac{N\big(\varphi(N+1)-\varphi(N)\big)}{\lfloor\varphi(N)\rfloor}\mathbb{E}_{n\in[N]}f^A_n(x)
\\
+\sum_{n=1}^{N-1}\frac{n\Big(\big(\varphi(n+1)-\varphi(n)\big)-\big(\varphi(n+2)-\varphi(n+1)\big)\Big)}{\lfloor \varphi(N)\rfloor}\mathbb{E}_{m\in[n]}f^A_m(x)
\\
=\frac{N\big(\varphi(N+1)-\varphi(N)\big)}{\lfloor\varphi(N)\rfloor}\mathbb{E}_{n\in[N]}f^A_n(x)
\\
+\sum_{n=N_h}^{N-1}\frac{n\Big(\big(\varphi(n+1)-\varphi(n)\big)-\big(\varphi(n+2)-\varphi(n+1)\big)\Big)}{\lfloor \varphi(N)\rfloor}\mathbb{E}_{m\in[n]}f^A_m(x)+O(\varphi(N)^{-1})
\text{,}
\end{multline}
where $N_h$ is the smallest natural number such that $\big(\varphi(n+1)-\varphi(n)\big)-\big(\varphi(n+2)-\varphi(n+1)\big)\ge0$ for all $n\ge N_h$. Taking into account the above together with $\eqref{zeropatternsconvenient}$ we obtain
\begin{multline}
0=\mathbb{E}_{x\in[N],n\in\mathbb{N}_h\cap[N]}1_A(x)1_A(x+n)1_A(x+2n)
=\mathbb{E}_{x\in[N]}\bigg[\frac{1}{\lfloor \varphi(N)\rfloor}\sum_{n\in[N]}\big(\varphi(n+1)-\varphi(n)\big)f_n^A(x)\bigg]
\\
+\mathbb{E}_{x\in[N]}\bigg[\sum_{n\in[N]}\frac{\Phi(-\varphi(n+1))-\Phi(-\varphi(n))}{\lfloor \varphi(N)\rfloor}f^{A}_n(x)\bigg]+O(\varphi(N)^{-1})
\\
=\frac{N\big(\varphi(N+1)-\varphi(N)\big)}{\lfloor \varphi(N)\rfloor}\mathbb{E}_{x,n\in[N]}[f_n^A(x)]+P_{h,A}(N)
\\
+\mathbb{E}_{x\in[N]}\bigg[\sum_{n\in[N]}\frac{\Phi(-\varphi(n+1))-\Phi(-\varphi(n))}{\lfloor \varphi(N)\rfloor}f^{A}_n(x)\bigg]+O(\varphi(N)^{-1})\text{,}
\end{multline}
where $P_{h,A}(N)\coloneqq\mathbb{E}_{n\in[N]}\Big[\sum_{n=N_h}^{N-1}\frac{n\big((\varphi(n+1)-\varphi(n))-(\varphi(n+2)-\varphi(n+1))\big)}{\lfloor \varphi(N)\rfloor}\mathbb{E}_{m\in[n]}f^A_m(x)\Big]$ which, crucially, is nonnegative. By Proposition~$\ref{KeyProp3}$ there exist positive constants $C=C(h)$ and  $\chi=\chi(h)<1/c$ such that
\begin{multline}\label{integererrorterm}
\Big|\mathbb{E}_{x\in[N]}\bigg[\sum_{n\in[N]}\frac{\Phi(-\varphi(n+1))-\Phi(-\varphi(n))}{\lfloor \varphi(N)\rfloor}f^{A}_n(x)\bigg]\Big|
\\
\lesssim\big|\mathbb{E}_{x\in[\pm 2N],n\in[N]}1_A(x)1_A(x+n)1_A(x+2n)w_N(n)\big|\lesssim CN^{-\chi}\text{.}
\end{multline}
Thus
\begin{equation}
0=\frac{N\big(\varphi(N+1)-\varphi(N)\big)}{\lfloor \varphi(N)\rfloor}\mathbb{E}_{x,n\in[N]}[f_n^A(x)]+P_{h,A}(N)+O(N^{-\chi})\text{,}
\end{equation}
and therefore
\begin{equation}
\frac{N\big(\varphi(N+1)-\varphi(N)\big)}{\lfloor \varphi(N)\rfloor}\mathbb{E}_{x,n\in[N]}[f_n^A(x)]+P_{h,A}(N)=O(N^{-\chi})\text{.}
\end{equation}
Both of the summands in the left-hand side are positive and thus we get
\begin{equation}
\frac{N\big(\varphi(N+1)-\varphi(N)\big)}{\lfloor \varphi(N)\rfloor}\mathbb{E}_{x,n\in[N]}[f_n^A(x)]=O(N^{-\chi})\text{,}
\end{equation}
and finally, since $\lim_{N\to\infty}\frac{N(\varphi(N+1)-\varphi(N))}{\lfloor \varphi(N)\rfloor}= \frac{1}{c}$, we have get
\begin{equation}
\mathbb{E}_{x,n\in[N]}1_A(x)1_A(x+n)1_A(x+2n)=O(N^{-\chi})\text{.}
\end{equation}
Let $\widetilde{A}\subseteq\mathbb{Z}_{2N+1}$ be the embedding of $A$ into $\mathbb{Z}_{2N+1}$, namely
\[
\widetilde{A}\coloneqq A+(2N+1)\mathbb{Z}\text{.}
\]
By identifying $\mathbb{Z}_{2N+1}$ with $\{0,\dotsc,(2N+1)-1\}$, where addition takes place modulo $2N+1$, we note that there exists a positive constant $C_h$ such that
\begin{multline}\label{MyEstimate}
\mathbb{E}_{x,n\in\mathbb{Z}_{2N+1}}1_{\widetilde{A}}(x)1_{\widetilde{A}}(x+n)1_{\widetilde{A}}(x+2n)
\\
=\frac{1}{(2N+1)^2}\sum_{x\in\mathbb{Z}_{2N+1},n\in\mathbb{Z}_{2N+1}\setminus\{0\}}1_{\widetilde{A}}(x)1_{\widetilde{A}}(x+n)1_{\widetilde{A}}(x+2n)+\frac{|\widetilde{A}|}{(2N+1)^2}
\\
\le\frac{2}{(2N+1)^2}\sum_{x,n\in[N]}1_{A}(x)1_{A}(x+n)1_{A}(x+2n)+\frac{1}{N}\le C_hN^{-\chi}\text{,}
\end{multline}
where for the penultimate estimate we note that for every non-trivial three-term arithmetic progression in $A$ considered on $\mathbb{Z}$ we have precisely two three-term arithmetic progression in $\widetilde{A}$ considered in $\mathbb{Z}_{2N+1}$. This is the reason we chose to embed our set in $\mathbb{Z}_{2N+1}$. It is clear that it suffices to establish the bound $\eqref{RothBound}$ for sufficiently large $N$; we will assume specifically that $N> \big(\frac{C_h}{C_0}\big)^{2/\chi}$, where $C_h$ and $\chi$ are the constants in the estimate $\eqref{MyEstimate}$, and $C_0$ is the constant from Theorem~$\ref{FancyRoth}$. By Theorem~$\ref{FancyRoth}$ we get that \begin{equation}\label{HonestBoundRoth}
|A|\le 6\alpha_3(N^{\chi/4})N\text{.}
\end{equation}
To see this, note that if we assume for the sake of contradiction that $|A|> 6\alpha_3(N^{\chi/4})N$, then
\[
\mathbb{E}_{x\in\mathbb{Z}_{2N+1}}1_{\widetilde{A}}(x)=\frac{|\widetilde{A}|}{2N+1}=\frac{|A|}{2N+1}> \frac{6\alpha_3(N^{\chi/4})N}{2N+1}=\frac{2\alpha_3(N^{\chi/4})3N}{2N+1}\ge 2\alpha_3(N^{\chi/4})\eqqcolon \alpha\text{,}
\]
and an application of Theorem~$\ref{FancyRoth}$ together with the estimate $\eqref{MyEstimate}$ yield
\[
C_hN^{-\chi}\ge \mathbb{E}_{x,n\in\mathbb{Z}_{2N+1}}1_{\widetilde{A}}(x)1_{\widetilde{A}}(x+n)1_{\widetilde{A}}(x+2n)\ge C_0 \big(\alpha_3^{-1}(\alpha/2)\big)^{-2}=C_0 \big(N^{\chi/4}\big)^{-2}=C_0 N^{-\chi/2}\text{,}
\]
and thus $\frac{C_h}{C_0}\ge N^{\chi/2}\iff N\le\big(\frac{C_h}{C_0}\big)^{2/\chi}$, which is a contradiction. Thus we have shown that for all sufficiently large $N$, we have
$|A|\le 6 \alpha_3(N^{\chi/4})N$, and by Theorem~$\ref{SOTAest}$ we get
\begin{equation}\label{calculationdensities}
|A|\le 6 \alpha_3(N^{\chi/4})N\le 6C_1\exp\big(-\tau(\log (N^{\chi/4}))^{1/9}\big)=6C_1\exp\big(-\tau(\chi/4)^{1/9}(\log N)^{1/9}\big)\text{,}
\end{equation}
which establishes the desired result. The proof is complete.
\end{proof}
For the sake of completeness we give a quick proof of Corollary~$\ref{RothFraPrimes}$. 
\begin{proof}[Proof of Corollary~$\ref{RothFraPrimes}$] It suffices to prove that any $A\subseteq \mathbb{P}$ of positive upper relative density contains at least one such pattern. Assume for the sake of contradiction that this is not the case, then by Theorem~$\ref{RothFra}$ and the prime number theorem we get
\begin{equation}\label{Primes3APfree}
\frac{|A\cap[N]|}{|\mathbb{P}\cap[N]|}\lesssim \frac{|A\cap[N]|}{N/\log N}\lesssim \log(N)\exp(-\chi(\log N)^{1/9})\lesssim \exp(-\chi/2(\log N)^{1/9})\text{,}
\end{equation}
and thus $\limsup_{N\to\infty}\frac{|A\cap[N]|}{|\mathbb{P}\cap[N]|}=0$, which is a contradiction. The proof is complete.
\end{proof}
\begin{proof}[Proof-sketch of Theorem~$\ref{corners}$] The proof is very similar to the proof of Theorem~$\ref{RothFra}$ so we only provide a sketch here. For every $N\in\mathbb{N}$ and $A\subseteq [N]^2$ lacking patterns of the form $(n_1,n_2),(n_1+\lfloor h(m)\rfloor,n_2),(n_1,n_2+\lfloor h(m)\rfloor)$, $n_1,n_2,m\in\mathbb{N}$ we may argue in a manner identical to before, using Proposition~$\ref{KeyProp1}$ in place of $\ref{KeyProp3}$ to conclude that there exist positive constants $\chi=\chi(h)<1$ and $C=C(h)$ such that
\begin{equation}\label{basicboundforcorners}
\mathbb{E}_{x\in[N]^2,n\in[N]}1_A(x)1_A(x+ne_1)1_A(x+ne_2)\le CN^{-\chi}\text{.}
\end{equation}
We now use the following corollary of Theorem~1.1 from \cite{cornerpaper}.
\begin{theorem}[Supersaturation for corners]\label{Sat}
There exists a positive constant $C'$ such that the following holds. Assume $N\in\mathbb{N}$, $A\subseteq [N]^2$ is of size $\delta N^2$ with $\delta\in(0,1/10)$ and $N>\exp\big(C'\log(1/\delta)^{600}\big)$.Then we have that $A$ contains at least $N^3\exp\big(-C'\log(1/\delta)^{600}\big)$  patterns of the form $(x,y),(x+d,y),(x,y+d)$, with $d>0$.
\end{theorem}
\begin{proof}
Deriving such a result from Theorem~1.1 in \cite{cornerpaper} is standard and we give a very rough sketch here. Firstly, one may obtain a bound of the form $C_0\exp\big(-c_0(\log N)^{1/600}\big)N^2$ for subsets of $[N]^2$ lacking patterns $(x,y),(x+d,y),(x,y+d)$ with $d>0$ by employing Theorem~1.1 from \cite{cornerpaper}, see the comments after the theorem in the aforementioned work. Having established such a result and using Varnavides’ standard argument \cite{VARN} one may immediately conclude. For details on how one may carry out this argument for corners we refer the reader to Appendix~F in \cite{BorysFappendix}.
\end{proof}
It suffices to establish the bound $\eqref{CornerBound}$ for $N\gtrsim 1$, and specifically we assume that $N\ge C^{2/\chi}$, where $C$ and $\chi$ are the constants appearing in $\eqref{basicboundforcorners}$. Let $\delta\coloneqq |A|/N^2$, and without loss of generality $\delta\in(0,1/10)$: if $\delta=0$, then there is nothing to prove, and if $\delta\in[1/10,1]$, we may perform the argument below viewing $A$ as a subset of $[10N+1]$.  If we assume that $N\le \exp\big(C'\log (1/\delta)^{600}\big)$, where $C'$ is as in Theorem~$\ref{Sat}$, then we immediately get
\begin{multline}
N\le \exp\big(C'\log(1/\delta)^{600}\big)\iff \big(\log (N)/C'\big)^{1/600}\le \log(1/\delta)
\\
\iff \log(\delta)\le -\big(\log (N)/C'\big)^{1/600} \iff |A|/N^2\le \exp\big(-(C')^{-1/600}\log(N)^{1/600}\big)\text{,}
\end{multline}
as desired. On the other hand, if we assume that $N>\exp\big(C'\log(1/\delta)^{600}\big)$, then Theorem~$\ref{Sat}$ is applicable and, by taking into account $\eqref{basicboundforcorners}$ as well as the fact that $N\ge C^{2/\chi}$, we obtain
\[
\exp\big(-C'\log(1/\delta)^{600}\big)\le \mathbb{E}_{x\in[N]^2,n\in[N]}1_A(x)1_A(x+ne_1)1_A(x+ne_2)=CN^{-\chi}\le N^{-\chi/2}\text{,}
\]
and thus
\begin{multline}
-C'\log(1/\delta)^{600}\le -(\chi/2)\log(N)\iff \log(1/\delta)^{600}\ge \chi\log(N)/(2C')
\\
\iff -\log(\delta)\ge \big( \chi\log(N)/(2C')\big)^{1/600} 
\iff |A|/N^2\le \exp\Big(- \Big(\frac{\chi}{2C'}\Big)^{1/600}(\log N)^{1/600}\Big)\text{.}
\end{multline}
Since $\chi\in(0,1)$, in either case, we get that 
\[
\frac{|A|}{N^2}\le \exp\Big(- \Big(\frac{\chi}{2C'}\Big)^{\frac{1}{600}}(\log N)^{\frac{1}{600}}\Big)\text{,}
\]
and the proof-sketch is complete.
\end{proof}
\begin{remark}\label{RemarkRoth}
We wish to make two brief remarks. Firstly, the proof of Theorem~$\ref{RothFra}$ establishes something slightly stronger, namely, it is proven that any $A\subseteq [N]$ lacking patterns of the form $n,n+\lfloor h(m)\rfloor,n+2\lfloor h(m)\rfloor$, $n,m\in\mathbb{N}$ satisfies the following density bound
\[
\frac{|A|}{N}\lesssim \alpha_3(N^{\chi'})\text{,}\quad\text{for some $\chi'=\chi'(h)>0$, see $\eqref{HonestBoundRoth}$.}
\]
Taking into account the calculation in $\eqref{calculationdensities}$, this implies that any improvement of the exponent $1/9$ of the logarithm in Theorem~$\ref{SOTAest}$ immediately yields an analogous improvement of the exponent in Theorem~$\ref{RothFra}$. Similarly, any improvement on the exponent $1/600$ for Theorem~1.1 in \cite{cornerpaper} yields an analogous improvement to Theorem~$\ref{corners}$.

Secondly, we also note that in the proof of Corollary~$\ref{RothFraPrimes}$, the estimate $\eqref{Primes3APfree}$ yields a quantitative variant of the result: the bound of the relative density of subsets $A\subseteq \mathbb{P}\cap[N]$ lacking such patterns is of the same form as $\eqref{RothBound}$.
\end{remark}

\end{document}